\documentclass{amsart}

\usepackage[OT2,T1]{fontenc}
\usepackage[utf8x]{inputenc}

\usepackage{enumitem}
\usepackage{color}
\usepackage{amssymb}
\usepackage{fourier-orns}
\usepackage{graphicx}
\usepackage[colorlinks=true,linkcolor=blue,citecolor=blue]{hyperref}
\usepackage{multirow}
\usepackage{dsfont}

\usepackage{titlesec}
\titlelabel{\thetitle. \ }

\usepackage{pgf,tikz}
\usetikzlibrary{arrows}
\definecolor{ffqqqq}{rgb}{1.,0.,0.}
\definecolor{qqzzff}{rgb}{0.,0.6,1.}
\usepackage{subcaption}

\newtheorem{theorem}{Theorem}[section]
\newtheorem{lemma}[theorem]{Lemma}
\newtheorem{proposition}[theorem]{Proposition}
\newtheorem{corollary}[theorem]{Corollary}

\theoremstyle{definition}

\newtheorem{remark}[theorem]{Remark}

\numberwithin{equation}{section}

\DeclareMathOperator{\re}{Re}
\DeclareMathOperator{\rank}{rank}
\DeclareMathOperator{\conv}{Conv}
\DeclareMathOperator{\diam}{diam}
\DeclareMathOperator{\vol}{Vol}

\begin{document}

\title[Wave equation on locally symmetric spaces]{Wave and Klein-Gordon equations \\ on certain locally symmetric spaces}

\author{Hong-Wei ZHANG}

\address{
Institut Denis Poisson (UMP 7013),
Université d’Orléans, Université de Tours \& CNRS,
Bâtiment de mathématiques 
- Rue de Chartres B.P. 6759 
- 45067 Orléans cedex 2  
- FRANCE
}

\email{hong-wei.zhang@univ-orleans.fr}

\thanks{This work is part of the author's Ph.D. thesis, supervised by J.-Ph. Anker and N. Burq.}

\thanks{The author thanks Marc Peigné for helpful hints about convex cocompact subgroups and Michel Marias for stimulating discussions.}



\begin{abstract}
This paper is devoted to study the dispersive properties of the linear Klein-Gordon and wave equations on a class of locally symmetric spaces. As a consequence, we obtain the Strichartz estimate and prove global well-posedness results for the corresponding semilinear equation with low regularity data as on real hyperbolic spaces.
\end{abstract}

\maketitle
\setcounter{tocdepth}{1}

\section{Introduction}
Let $M$ be a Riemannian manifold and denote by $\Delta$ the Laplace-Beltrami operator on $M$. The theory is well established for the following wave equation on $M= \mathbb{R}^n$,
\begin{align}\label{wave equation}
\begin{cases}
\partial_{t}^2 u(t,x)- \Delta u(t,x) = F(t,x), \\
u(0,x) =f(x) , \ \partial_{t}|_{t=0} u(t,x) =g(x),
\end{cases}
\end{align}
where the solutions $u$ satisify the Strichartz estimates:
\footnote{The symbol $\lesssim$, let us recall, means precisely that there exists a constant $0 < C < + \infty$ such that $\left\| \nabla_{\mathbb{R} \times \mathbb{R}^n} u \right\|_{L^p(I; H^{-\sigma,q}(\mathbb{R}^n))} \le C \left(  \left\| f \right\|_{H^1(\mathbb{R}^n)} + \left\| g \right\|_{L^2(\mathbb{R}^n)} + \left\| F \right\|_{L^{\tilde{p}'}(I; H^{\tilde{\sigma},\tilde{q}'}(\mathbb{R}^n))} \right)$, where $\tilde{p}'$ is the dual exponent of $\tilde{p}$, defined by the formula $\frac{1}{\tilde{p}} + \frac{1}{\tilde{p}'} = 1$, so is $\tilde{q}'$.}
\begin{align*}
\left\| \nabla_{\mathbb{R} \times \mathbb{R}^n} u \right\|_{L^p(I; H^{-\sigma,q}(\mathbb{R}^n))} \lesssim \left\| f \right\|_{H^1(\mathbb{R}^n)} + \left\| g \right\|_{L^2(\mathbb{R}^n)} + \left\| F \right\|_{L^{\tilde{p}'}(I; H^{\tilde{\sigma},\tilde{q}'}(\mathbb{R}^n))},
\end{align*}
on any interval $I \subseteq \mathbb{R}$ under the assumptions
\begin{align*}
\sigma = \frac{n+1}{2} \left( \frac{1}{2} - \frac{1}{q} \right), \ \widetilde{\sigma} = \frac{n+1}{2} \left( \frac{1}{2} - \frac{1}{\tilde{q}} \right),
\end{align*}
and the couples $(p,q), (\tilde{p}, \tilde{q}) \in (2, +\infty ] \times [2, 2 \frac{n-1}{n-3})$ fulfill the admissibility conditions:
\begin{align*}
\frac{1}{p} = \frac{n-1}{2} \left( \frac{1}{2} - \frac{1}{q} \right), \ \frac{1}{\tilde{p}} = \frac{n-1}{2} \left( \frac{1}{2} - \frac{1}{\tilde{q}} \right).
\end{align*}

These estimates serve as a tool for finding minimal regularity conditions on the initial data ensuring well-posedness for corresponding semilinear wave equations, which is addressed in \cite{kapitanski1994weak}, and almost fully answered in \cite{dancona2001weighted, georgiev1997weighted, keel1998endpoint, lindblad1995existence}.

Analogous results have been found for the Klein-Gordon equation 
\begin{align}\label{KG equation}
\begin{cases}
\partial_{t}^2 u(t,x)- \Delta u(t,x) + cu(t,x)= F(t,x), \\
u(0,x) =f(x) , \ \partial_{t}|_{t=0} u(t,x) =g(x).
\end{cases}
\end{align}
with $c=1$, see \cite{bahouri1999high, ginibre1995generalized, machihara2004small, nakanishi1999scattering}.

Given the rich Euclidean theory, it is natural to look at the corresponding equations on more general manifolds. We consider in the present paper a class of noncompact locally symmetric spaces $M$, on which we study the Klein-Gordon equation \eqref{KG equation} with $c \ge -\rho^2$, where $\rho$ is a positive constant depending on the structure of $M$ and defined in the next section. Due to large-scale dispersive effects in negative curvature, we expect stronger results than in the Euclidean setting, as on real hyperbolic space, see \cite{anker2014wave, anker2012wave}.

In the critical case $c = -\rho^2$, \eqref{KG equation} is called the shifted wave equation. To our knowledge, it was first considered in \cite{fontaine1994equation, fontaine1997semilinear} in low dimensions $n=2$ and $n=3$. In \cite{anker2012wave, anker2015wave}, a detailed analysis of the shifted wave equation was carried out on real hyperbolic spaces and on Damek-Ricci spaces, which contains all rank one symmetric spaces of noncompact type. In the non-shifted case $c > -\rho^2$, similar results on real hyperbolic spaces were obtained in \cite{anker2014wave}.

In the recent paper \cite{fotiadis2017schrodinger}, the Schrödinger equation was  considered on certain locally symmetric spaces. In the present paper, we study the wave and Klein-Gordon equations in the same spirit.

\subsection{Notations} { \ } \par
We adopt the standard notation (see for instance \cite{helgason1984groups}, \cite{bunke2000spectrum}). Let $G$ be a semisimple Lie group, connected, noncompact, with finite center, and $K$ be a maximal compact subgroup of $G$. The homogenous space $X=G/K$ is a Riemannian symmetric space of  noncompact type, whose dimension is denoted by $n$. Let $\mathfrak{g} = \mathfrak{k} \oplus \mathfrak{p}$ be the Cartan decomposition of its Lie algebra. The Killing form of $\mathfrak{g}$ induces a $K$-invariant inner product on $\mathfrak{p}$, and hence a $G$-invariant Riemannian metric on $G/K$. 

Fix a maximal abelian subspace $\mathfrak{a}$ in $\mathfrak{p}$. The symmetric space $X$ is said to have rank one if $\dim \mathfrak{a} = 1$. Denote by $\mathfrak{a}^*$ the real dual of $\mathfrak{a}$, let $\Sigma \subset \mathfrak{a}^*$ be the root system of $(  \mathfrak{g}, \mathfrak{a})$ and denote by $W$ the Weyl group associated to $\Sigma$. Choose a set $\Sigma^{+}$ of positive roots, let $\mathfrak{a}^{+} \subset \mathfrak{a}$ be the corresponding positive Weyl chamber and $\overline{\mathfrak{a}^{+}}$ its closure. Denote by $\rho$ the half sum of positive roots counted with their multiplicities:
\begin{align*}
\rho = \frac{1}{2} \sum_{\alpha \in \Sigma^{+}} m_{\alpha} \alpha ,
\end{align*}
where $m_{\alpha}$ is the dimension of root space $\mathfrak{g}_{\mathfrak{a}} = \left\lbrace Y \in \mathfrak{g} \ | \left[ {H,Y } \right] = \alpha(H)Y, \forall H \in \mathfrak{a} \ \right\rbrace$.

Let $\Gamma$ be a discrete torsion-free subgroup of $G$. The locally symmetric space $M=\Gamma \backslash X$, equipped with the Riemannian structure inherited from $X$ becomes a Riemannian manifold. We say that $M$ has rank one if $X$ has rank one. Moreover $\Gamma$ is called convex cocompact if the quotient group $\Gamma \backslash \conv (\Lambda_{\Gamma})$ is compact, where $\conv (\Lambda_{\Gamma})$ is the convex hull of the limit set $\Lambda_{\Gamma}$ of $\Gamma$. We denote by $\Delta$ the Laplace-Beltrami operator, by $d(\cdot , \cdot )$ the Riemannian distance, and by $dx$ the associated measure, both on $X$ and $M$. Consider the Poincaré series
\begin{align*}
P(s;x,y) = \sum_{\gamma \in \Gamma} e^{-s d(x, \gamma y)}, \ s>0, \ x,y \in X,
\end{align*}
and denote by $\delta(\Gamma)$ its critical exponent:
\begin{align*}
\delta(\Gamma) = \inf \left\lbrace s>0 \ | \ P(s;x,y) < + \infty \right\rbrace .
\end{align*}

\subsection{Assumptions}\label{subsection assumptions} { \ } \par
In this paper, $M=\Gamma \backslash X$ is a rank one locally symmetric space such that $\Gamma$ is convex cocompact and $\delta(\Gamma) < \rho$.

Let us comment a few words on these assumptions. Wave type equations on noncompact rank one symmetric spaces are well understood. Sharp pointwise estimates of wave kernels on $X$(see Section \ref{subsection wave kernel}), which were obtained in \cite{ anker2014wave, anker2012wave}, allow us to deal with wave kernels on a locally symmetric space $M$. Notice that such information is lacking in higher rank.

The rank one symmetric spaces of the noncompact type are the hyperbolic spaces $H^n({\mathbb{F}})$ with $\mathbb{F} = \mathbb{R}, \mathbb{C}, \mathbb{H}$ or $H^2({\mathbb{O}})$. In particular, we have $\mathfrak{a}^{*} = \mathfrak{a}$ and $\mathfrak{a}^{+} \cong \mathbb{R}_{+}^{*}$, hence $\rho$ is just a positive constant depending on the structure of $X$. Specifically, as a direct consequence of the assumption $\delta(\Gamma) < \rho$, the series \eqref{wave kernel on M} defining the wave kernel on $M$ is absolutely convergent, see Proposition \ref{proposition wave kernel on M}. In addition, according to \cite{corlette1990hausdorff}, the bottom $\lambda_0$ of the $L^2$-spectrum of $-\Delta$ on $M$ is equal to $\rho^2$, as on $X$. Consequently, we obtain an analogous $L^2$ Kunze-Stein phenomenon on $M$ without further assumptions, see Proposition \ref{Kunze-Stein}. Notice that $\lambda_0 = \rho^2 > 0$ implies $\vol (M) = + \infty$, while $\lambda_0  = 0$ if $M$ is a lattice.

At last, the convex cocompactness assumption implies a uniform upper bound of the Poincaré series, see Lemma \ref{estimate of Poincaré series}, which is crucial for the $L^1 \rightarrow L^{\infty}$ boundedness of wave propagators on $M$.

\begin{remark}
The Schrödinger equation is studied in \cite{fotiadis2017schrodinger} under slightly different assumptions, our well-posedness results hold also in that setting.
\end{remark}
\subsection{Statement of the results} { \ } \par
Consider the operator $D= \sqrt{-\Delta - \rho^2 + \kappa^2}$ with $\kappa >0$, then the Klein-Gordon equations \eqref{KG equation} becomes
\begin{align}\label{KG equation D}
\begin{cases}
\partial_{t}^2 u(t,x) + D_{x}^2 u(t,x) = F(t,x), \\
u(0,x) =f(x) , \ \partial_{t}|_{t=0} u(t,x) =g(x).
\end{cases}
\end{align}
with $c = \kappa^2 - \rho^2 > - \rho^2$. Notice that \eqref{KG equation D} is the wave equation when $\kappa = \rho$ and becomes the shifted wave equation in the limit case $ \kappa = 0$. Consider another operator $\widetilde{D} = \sqrt{-\Delta - \rho^2 + \widetilde{\kappa}^2}$ with $\widetilde{\kappa} > \rho$. We denote by $\omega_{t}^{\sigma}$ the radial convolution kernel of the wave operator $ W_{t}^{\sigma}:= \widetilde{D}^{-\sigma} e^{it D}$ on the symmetric space $X$:
\begin{align}\label{wave kernel on X}
W_{t}^{\sigma} f(x) = f*\omega_{t}^{\sigma}(x) = \int_{G} \omega_{t}^{\sigma}(y^{-1}x)f(y) dy,
\end{align}
where $f$ is any reasonable function on $X$, see Section \ref{subsection wave kernel} for more details. By $K$-bi-invariance of the kernel $\omega_{t}^{\sigma}$, we deduce that $W_{t}^{\sigma} f$ is left $\Gamma$-invariant and right $K$-invariant if $f$ is defined on the locally symmetric space $M$. Thus the wave operator on $M$, denoted by $\widehat{W_{t}^{\sigma}}$, is also defined by \eqref{wave kernel on X}. Consider the wave kernel $\widehat{\omega_{t}^{\sigma}}$ on $M$, which is given by
\begin{align*}
\widehat{\omega_{t}^{\sigma}}(x,y) = \sum_{\gamma \in \Gamma} \omega_{t}^{\sigma} (y^{-1} \gamma x), \ \forall x,y \in X.
\end{align*}
Then the wave operator $\widehat{W_{t}^{\sigma}}$ on $M$ is an integral operator:
\begin{align*}
\widehat{W_{t}^{\sigma}} f(x) = \int_{M} \widehat{\omega_{t}^{\sigma}}(x,y) f(y) dy,
\end{align*}
see Proposition \ref{wave kernel on M}. The aim of this paper is to prove the following dispersive properties:
\begin{theorem}\label{dispersive estimate}
For $n \ge 3$, $2 < q< +\infty$ and $\sigma \ge (n+1) \left( \frac{1}{2} - \frac{1}{q} \right)$,
\begin{align}
\left\| \widehat{W_{t}^{\sigma}} \right\|_{L^{q'}(M) \rightarrow L^q(M)} \lesssim
\begin{cases}
|t|^{-(n-1)\left( \frac{1}{2} - \frac{1}{q} \right)} \quad &if \ 0< |t| < 1, \\
|t|^{-\frac{3}{2}} \quad &if \ |t| \ge 1.
\end{cases}
\end{align}
\end{theorem}

\begin{remark}
In dimension $n=2$, there is an additional logarithmic factor in the small time bound, which becomes $|t|^{- \left( \frac{1}{2} - \frac{1}{q} \right)} (1- \log |t| )^{1-\frac{2}{q}}$, see Theorem \ref{pointwise kernel estimate} in the next section.
\end{remark}

\begin{remark}
At the endpoint $q=2$, $t \mapsto e^{itD}$ is a one-parameter group of unitary operators on $L^2(M)$.
\end{remark}

By applying the classical $TT^*$ method and by using the previous dispersive properties, we obtain the Strichartz estimate
\begin{align*}
\left\| \nabla_{\mathbb{R} \times M} u \right\|_{L^p(I; H^{-\sigma,q}(M))} \lesssim \left\| f \right\|_{H^1(M)} + \left\| g \right\|_{L^2(M)} + \left\| F \right\|_{L^{\tilde{p}'}(I; H^{\tilde{\sigma},\tilde{q}'}(M))}
\end{align*}
for the solutions $u$ of \eqref{KG equation D}, see Section \ref{section applcations} for more information about the Sobolev spaces $H^{-\sigma,q}(M)$. Here $I \subset \mathbb{R}$ is any time interval, possibly unbounded, 
\begin{align*}
\sigma \ge \frac{n+1}{2} \left( \frac{1}{2} - \frac{1}{q} \right), \ \widetilde{\sigma} \ge \frac{n+1}{2} \left( \frac{1}{2} - \frac{1}{\tilde{q}} \right),
\end{align*}
and the couples $(p,q)$ and $(\tilde{p}, \tilde{q})$ are admissible, which means that $\left( \frac{1}{p}, \frac{1}{q} \right)$, $\left( \frac{1}{\tilde{p}}, \frac{1}{\tilde{q}} \right)$ belong, in dimension $n \ge 4$ (see Section \ref{section applcations} for the lower dimensions) to the triangle
\begin{align*}
\left\lbrace \left( \frac{1}{p}, \frac{1}{q} \right) \in \left( 0, \frac{1}{2} \right) \times \left( 0, \frac{1}{2} \right) \ \Big| \ \frac{1}{p} \ge \frac{n-1}{2} \left( \frac{1}{2} - \frac{1}{q} \right) \right\rbrace \bigcup \left\lbrace  \left( 0, \frac{1}{2} \right),  \left( \frac{1}{2}, \frac{1}{2} - \frac{1}{n-1} \right) \right\rbrace.
\end{align*}
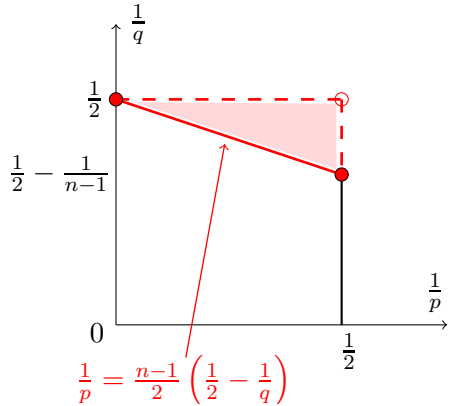
\begin{figure}[h]
\centering
\definecolor{ffqqqq}{rgb}{1.,0.,0.}
\begin{tikzpicture}[scale=0.5][line cap=round,line join=round,>=triangle 45,x=1.0cm,y=1.0cm]
\draw[->,color=black] (0.,0.) -- (8.805004895364174,0.);
\foreach \x in {,2.,4.,6.,8.,10.}
\draw[shift={(\x,0)},color=black] (0pt,-2pt);
\draw[color=black] (8,0.08808504628984362) node [anchor=south west] {\large $\frac{1}{p}$};
\draw[->,color=black] (0.,0.) -- (0.,8);
\foreach \y in {,2.,4.,6.,8.}
\draw[shift={(0,\y)},color=black] (-2pt,0pt);
\draw[color=black] (0.11010630786230378,8) node [anchor=west] {\large $\frac{1}{q}$};
\clip(-5,-2.1) rectangle (11.805004895364174,9.404473957611293);
\fill[line width=2.pt,color=ffqqqq,fill=ffqqqq,fill opacity=0.15000000596046448] (0.45855683542994374,5.928142080369191) -- (5.85410415683891,5.8847171522290775) -- (5.864960388873938,4.136863794589543) -- cycle;
\draw [line width=1.pt,dash pattern=on 5pt off 5pt,color=ffqqqq] (0.,6.)-- (6.,6.);
\draw [line width=1.pt,dash pattern=on 5pt off 5pt,color=ffqqqq] (6.,6.)-- (6.,4.);
\draw [line width=1.pt,color=ffqqqq] (6.,4.)-- (0.,6.);
\draw [line width=0.8pt] (6.,4.)-- (6.,0.);
\draw [->,line width=0.5pt,color=ffqqqq] (1.8569381602310457,-0.8701903243695805) -- (2.891937454136701,4.811295161325333);
\begin{scriptsize}
\draw [fill=ffqqqq] (0.,6.) circle (5pt);
\draw[color=black] (-0.5709441083587729,6) node {\large $\frac{1}{2}$};
\draw[color=black] (-1.5,4) node {\large $\frac{1}{2}- \frac{1}{n-1}$};
\draw[color=black] (-0.5,-0.2) node {\large $0$};
\draw[color=black] (6.2,-0.7) node {\large $\frac{1}{2}$};
\draw[color=ffqqqq] (1.8,-1.5) node {\large $\frac{1}{p} = \frac{n-1}{2} \left( \frac{1}{2} - \frac{1}{q} \right)$};
\draw [color=ffqqqq] (6.,6.) circle (5pt);
\draw [fill=ffqqqq] (6.,4.) circle (5pt);
\end{scriptsize}
\end{tikzpicture}
\caption{Admissibility in dimension $n \ge 4$.}
\end{figure}

Notice that the admissible set for $M$ is larger than the admissible set for $\mathbb{R}^n$ which corresponds only to the lower edge of the triangle. In comparison with $X$, we loose the right edge of the triangle, which corresponds to the critical case $\frac{1}{p} = \frac{1}{2}$ and $\frac{1}{q} > \frac{1}{2} - \frac{1}{n-1}$, this will be explained in Section \ref{section applcations}. Notice that we obtain nevertheless the same well-posedness results as on $X$.

This paper is organized as follows. In Section \ref{section preliminairies}, we review spherical analysis on noncompact symmetric spaces, and recall pointwise estimates of wave kernels on rank one symmetric space obtained in \cite{anker2014wave}. In Section \ref{section main section}, after proving the necessary lemmas, we prove the dispersive estimate by an interpolation argument. As a consequence, we deduce the Strichartz estimate and obtain well-posedness results for the semilinear Klein-Gordon equation in Section \ref{section applcations}.

\section{Preliminairies}\label{section preliminairies}

\subsection{Spherical analysis on noncompact symmetric spaces}\label{subsection spherical analysis} { \ } \par

We review in this section some elementary facts about noncompact symmetric spaces. We refer to \cite{anker1990Lp, anker1996spherical, faraut1982analyse, helgason1984groups} for more details. 

Recall that $\overline{\mathfrak{a}^{+}}$ is the closure of the positive Weyl chamber $\mathfrak{a}^{+}$. Denote by $\mathfrak{n} = \sum_{\alpha \in \Sigma^{+}} \mathfrak{g}_{\alpha}$ the nilpotent Lie subalgebra of $\mathfrak{g}$ associated with $\Sigma^{+}$, and by $N$ the corresponding Lie subgroup of $G$. Then we have the following two decompositions of $G$:
\begin{align*}
\begin{cases}
G= N \left( \exp \mathfrak{a} \right) K \quad &(Iwasawa), \\
G= K \left( \exp \overline{\mathfrak{a}{+}} \right) K \quad &(Cartan).
\end{cases}
\end{align*}

In the Cartan decomposition, the Haar measure on $G$ writes
\begin{align}\label{Cartan decomposition for Haar measure}
\int_G f(g) dg = const. \int_{K} dk_1 \int_{\mathfrak{a}_{+}} \prod_{\alpha \in \Sigma^{+}} \left( \sinh \alpha(H)  \right )^{m_{\alpha}}  dH \int_{K} f(k_1 (\exp H) k_2 ) dk_2.
\end{align}
In the rank one case, which we consider in this paper,
\begin{align*}
\int_{\mathfrak{a}_{+}} \prod_{\alpha \in \Sigma^{+}} \left( \sinh \alpha(H)  \right )^{m_{\alpha}} dH = const. \int_{0}^{+ \infty} \left( \sinh r \right)^{m_{\alpha}} \left( \sinh 2r \right)^{m_{2 \alpha}} dr,
\end{align*}
where
\begin{align}\label{estimate of Haar measure}
\left( \sinh r \right)^{m_{\alpha}} \left( \sinh 2r \right)^{m_{2 \alpha}} \lesssim e^{2 \rho r}, \quad \forall r>0.
\end{align}

Denote by $\mathcal{S}(K \backslash G /K)$ the Schwartz space of $K$-bi-invariant functions on $G$. The spherical Fourier transform $\mathcal{H}$ is defined by
\begin{align*}
\mathcal{H} f (\lambda) = \int_{G} f(x) \varphi_{\lambda} (x) dx, \ \forall \lambda \in \mathfrak{a}^{*} \cong \mathbb{R}, \ \forall f \in \mathcal{S} (K \backslash G/K).
\end{align*}
Here $\varphi_{\lambda} \in \mathcal{C}^{\infty} (K \backslash G/K)$ is a spherical function, which can be characterized as a radial eigenfunction of the negative Laplace-Beltrami operator  $- \Delta$ satisfying 
\begin{equation}\label{eigenfunction}
\begin{cases}
- \Delta \varphi_{ \lambda} (x) = \left(  \lambda^2 + \rho^2 \right) \varphi_{\lambda}(x), \\
\varphi_{\lambda}(e) =1.
\end{cases}
\end{equation}
In the noncompact case, the spherical function is characterized by
\begin{align}\label{spherical function on noncompact spaces}
\varphi_{\lambda} (x) = \int_{K} e^{(i \lambda + \rho)A(kx)} dk, \ \lambda \in {\mathfrak{a}_{\mathbb{C}}^*},
\end{align}
where $A(kx)$ is the unique $\mathfrak{a}$-component in the Iwasawa decomposition of $kx$.

Denote by $\mathcal{S} \left( \mathfrak{a}^* \right)^{W}$ the subspace of $W$-invariant functions in the Schwartz space $\mathcal{S} \left( \mathfrak{a}^* \right)$. Then $\mathcal{H}$ is an isomorphism between $\mathcal{S}(K \backslash G /K)$ and $\mathcal{S} \left( \mathfrak{a}^* \right)^{W}$. The inverse spherical Fourier transform is defined by
\begin{align*}
f (x) = const. \int_{\mathfrak{a}^{*}} \mathcal{H}f(\lambda) \varphi_{- \lambda} (x) |\mathbf{c(\lambda)}|^{-2} d \lambda, \ \forall x \in G , \ \forall f  \in \mathcal{S} (\mathfrak{a}^{*})^{W},
\end{align*}
where $\mathbf{c(\lambda)}$ is the Harish-Chandra $\mathbf{c}$-function. 

\subsection{Pointwise estimates of the wave kernel on symmetric spaces}\label{subsection wave kernel} { \ } \par
We recall in this section the pointwise wave kernel estimates on rank one symmetric space obtained in \cite{anker2014wave} and \cite{anker2015wave}. Via the spherical Fourier transform and \eqref{eigenfunction}, the negative Laplace-Beltrami operator $- \Delta$ corresponds to $\lambda^2 + \rho^2$, hence the operators $D= \sqrt{-\Delta - \rho^2 + \kappa^2}$ and $\widetilde{D} = \sqrt{-\Delta - \rho^2 + \widetilde{\kappa}^2}$ to
\begin{align*}
\sqrt{\lambda^2 + \kappa^2} \quad {and} \quad \sqrt{\lambda^2 + \tilde{\kappa}^2}.
\end{align*}
By the inverse spherical Fourier transform, the radial convolution kernel $\omega_{t}^{\sigma}$ of $ W_{t}^{\sigma}= \widetilde{D}^{-\sigma} e^{it D}$ on $X$ is given by
\begin{align*}
\omega_{t}^{\sigma}(r) = const. \int_{- \infty}^{+ \infty} (\lambda^2 + \tilde{\kappa}^2 )^{-\frac{\sigma}{2}} e^{it \sqrt{\lambda^2 + \kappa^2}} \varphi_{\lambda}(r) |\mathbf{c} (\lambda) |^{-2} d \lambda
\end{align*}
for suitable exponents $\sigma \in \mathbb{R}$. Consider smooth even cut-off functions $\chi_0$ and $\chi_{\infty}$ on $\mathbb{R}$ such that
\begin{align*}
\begin{cases}
\chi_{0}(\lambda) + \chi_{\infty}(\lambda) =1, \\
\chi_{0}(\lambda) = 1, \quad \forall |\lambda| \le 1, \\
\chi_{\infty}(\lambda) =1, \quad \forall |\lambda| \ge 2.
\end{cases}
\end{align*}
Let us split up
\begin{align*}
\omega_{t}^{\sigma} (r) =& \omega_{t}^{\sigma,0} (r) + \omega_{t}^{\sigma, \infty} (r) \\
=& \ const. \int_{- \infty}^{+ \infty} \chi_{0}(\lambda)  (\lambda^2 + \tilde{\kappa}^2 )^{-\frac{\sigma}{2}} e^{it \sqrt{\lambda^2 + \kappa^2}} \varphi_{\lambda}(r) |\mathbf{c} (\lambda) |^{-2} d \lambda \\
+& \ const. \int_{- \infty}^{+ \infty} \chi_{\infty}(\lambda) (\lambda^2 + \tilde{\kappa}^2 )^{-\frac{\sigma}{2}} e^{it \sqrt{\lambda^2 + \kappa^2}} \varphi_{\lambda}(r) |\mathbf{c} (\lambda) |^{-2} d \lambda.
\end{align*}
To avoid possible singularities of the kernel $\omega_{t}^{\sigma, \infty}$, see \cite[Chap 9]{stein1993harmonic}, we consider the analytic family of operators
\begin{align}\label{analytic operator}
\widetilde{W}_{t}^{\sigma, \infty} := \frac{e^{\sigma^2}}{\Gamma \left( \frac{n+1}{2} - \sigma \right)} \chi_{\infty} (D) \tilde{D}^{-\sigma} e^{it D},
\end{align}
in the vertical strip $0 \le \re \sigma \le \frac{n+1}{2}$, and their kernels
\begin{align*}
\widetilde{\omega}_{t}^{\sigma, \infty}(r) = const.  \frac{e^{\sigma^2}}{\Gamma \left( \frac{n+1}{2} - \sigma \right)} \int_{- \infty}^{+ \infty} \chi_{\infty}(\lambda) (\lambda^2 + \tilde{\kappa}^2 )^{-\frac{\sigma}{2}} e^{it \sqrt{\lambda^2 + \kappa^2}} \varphi_{\lambda}(r) |\mathbf{c} (\lambda) |^{-2} d \lambda.
\end{align*}

The following pointwise estimates of the kernels $\omega_{t}^{\sigma,0}$ and $\widetilde{\omega}_{t}^{\sigma, \infty}$, which were obtained in \cite{anker2014wave} for real hyperbolic spaces, extend straightforwardly to all rank one Riemannian symmetric spaces of the noncompact type.

\begin{theorem}\label{pointwise kernel estimate}
For all $\sigma \in \mathbb{R}$, the kernel $\omega_{t}^{\sigma,0}$ satisfies
\begin{align*}
|\omega_{t}^{\sigma,0}(r)| \lesssim
\begin{cases}
\varphi_0 (r), \quad &\forall t\in \mathbb{R}, \ \forall r \ge 0, \\
|t|^{- \frac{3}{2}}(1+r) \varphi_0 (r), &\forall |t| \ge 1, \ \forall \ 0 \le r \le  \frac{|t|}{2}.
\end{cases}
\end{align*}
For all $\sigma \in \mathbb{C}$ with $\re \sigma = \frac{n+1}{2}$, and for every $r \ge 0$, the following estimates hold for the kernel $\widetilde{\omega}_{t}^{\sigma, \infty}$: 
\begin{align*}
|\widetilde{\omega}_{t}^{\sigma, \infty}(r)| \lesssim
\begin{cases}
|t|^{- \frac{n-1}{2}} e^{- \rho r}, \quad &\forall 0 < |t|<1,  \ if \ n \ge 3, \\
|t|^{- N} (1+r)^{N} \varphi_{0}(r), \quad &\forall |t| \ge 1, \ \forall N \in \mathbb{N}.
\end{cases}
\end{align*}
In the $2$-dimensional case, the small time estimate of  $\widetilde{\omega}_{t}^{\sigma, \infty}$ reads
\begin{align*}
| \widetilde{\omega}_{t}^{\sigma, \infty} (r) | \lesssim |t|^{-\frac{1}{2}} (1 - \log |t| ) e^{- \frac{r}{2}}, \quad \forall 0 < |t|<1.
\end{align*}
\end{theorem}
\section{Dispersive estimates for the wave operator \\ on locally symmetric spaces}\label{section main section} 
In this section, we prove our main result, namely Theorem \ref{dispersive estimate}. Let us first describe the wave operator $\widehat{W_{t}^{\sigma}}$ on locally symmetric space $M$. Recall that the wave kernel on $M$ is given by 
\begin{align}\label{wave kernel on M}
\widehat{\omega_{t}^{\sigma}}(x,y) = \sum_{\gamma \in \Gamma} \omega_{t}^{\sigma} (y^{-1} \gamma x), \ \forall x,y \in X.
\end{align}
\begin{proposition}\label{proposition wave kernel on M}
The series \eqref{wave kernel on M} is convergent for every $x,y \in X$, and the wave operator on $M$ is given by
\begin{align*}
\widehat{W_{t}^{\sigma}} f(x) = \int_{M} \widehat{\omega_{t}^{\sigma}}(x,y) f(y) dy,
\end{align*}
for any reasonable function $f$ on $M$.
\end{proposition}

\begin{proof}
According to the Cartan decomposition of $G$, we can write $y^{-1} \gamma x = k_{\gamma}( \exp H_{\gamma} )k_{\gamma}'$ with $H_{\gamma} \in \mathbb{R}_{+}$ and $k_{\gamma}, k_{\gamma}' \in K$. Notice that $H_{\gamma} = d(x, \gamma y)$. Then, by the $K$-bi-invariance of $\omega_{t}^{\sigma}$, we have
\begin{align*}
| \widehat{\omega_{t}^{\sigma}} (x,y) | 
= \left| \sum_{\gamma \in \Gamma} \omega_{t}^{\sigma} ( \exp H_{\gamma}) \right| \lesssim \sum_{\gamma \in \Gamma} |  \omega_{t}^{\sigma,0} (\exp H_{\gamma}) | +  \sum_{\gamma \in \Gamma} | \tilde{\omega}_{t}^{\sigma, \infty} (\exp H_{\gamma}) |.
\end{align*}

For the first part, Theorem \ref{pointwise kernel estimate} implies that for all $H_{\gamma} \ge 0$, 
\begin{align*}
\sum_{\gamma \in \Gamma} |  \omega_{t}^{\sigma,0} (\exp H_{\gamma}) | \lesssim \sum_{\gamma \in \Gamma}  \varphi_0 (\exp H_{\gamma}) 
\lesssim \sum_{\gamma \in \Gamma} (1+H_{\gamma}) e^{-\rho H_{\gamma}}.
\end{align*}

By choosing $0< \varepsilon < \rho - \delta(\Gamma)$, we obtain
\begin{align*}
\sum_{\gamma \in \Gamma} | \omega_{t}^{\sigma,0} (\exp H_{\gamma}) | \lesssim  \sum_{\gamma \in \Gamma} e^{- (\delta(\Gamma) + \varepsilon) d(x, \gamma y)} = P_{\delta(\Gamma) + \varepsilon} (x,y) < + \infty.
\end{align*}
The second part is handled similarly and thus omitted. Hence the serie \eqref{wave kernel on M} is convergent. According to \eqref{wave kernel on X}, we know that
\begin{align*}
\widehat{W_{t}^{\sigma}} f(x) =  \int_{G} \omega_{t}^{\sigma}(y^{-1}x)f(y) dy = \int_{X} \omega_{t}^{\sigma}(y^{-1}x)f(y) dy,
\end{align*}
since $\omega_{t}^{\sigma}$ is $K$-bi-invariant and $f$ is right $K$-invariant. By using Weyl's formula and the fact that $f$ is left $\Gamma$-invariant, we deduce that
\begin{align*}
\widehat{W_{t}^{\sigma}} f(x) = \int_{\Gamma \backslash X} \left( \sum_{\gamma \in \Gamma} f( \gamma y) \omega_{t}^{\sigma} (y^{-1} \gamma x) \right) dy  = \int_{M} \widehat{\omega_{t}^{\sigma}}(x,y) f(y) dy.
\end{align*}
\end{proof}

Next, we introduce the following version of the $L^2$ Kunze-Stein phenomenon on locally symmetric space $M$, which plays an essential role in the proof of the dispersive estimate.

\begin{proposition}\label{Kunze-Stein}
Let $\psi$ be a reasonable bi-$K$-invariant functions on $G$, e.g., in the Schwartz class. Then
\begin{align}
\| . * \psi \|_{L^2(M) \rightarrow L^2(M)} \le \int_{G} | \psi (x)| \varphi_{0}(x) dx.
\end{align}
\end{proposition}

The Kunze-Stein phenomenon is a remarkable convolution property on semisimple Lie groups and symmetric spaces (see e.g.,  \cite{kunze1960uniformly}, \cite{herz1970kunze}, \cite{cowling1978kunze} and \cite{ionescu2000endpoint}), which was extended to some classes of locally symmetric spaces in \cite{lohoue2009invariants} and \cite{lohoue2014multipliers}. Let us prove Proposition \ref{Kunze-Stein} along the lines of \cite{lohoue2014multipliers} in our setting where $\rank X = 1$, $\delta(\Gamma) < \rho$ and with no additional assumption.

\begin{proof}
Since $G$ is a connected semisimple Lie group, it is of type I, \cite{harish1954representations, dixmier1977c*}. Denote by $\widehat{G}$ the unitary dual of $G$ and by $\widehat{G}_K$ the spherical subdual. We write $L^2 (\Gamma \backslash G )$ as the direct integral
\begin{align*}
L^2 (\Gamma \backslash G ) \cong \int_{\widehat{G}}^{\oplus} \mathcal{H}_{\pi} d \nu (\pi)
\end{align*}
and
\begin{align}\label{direct integral}
L^2 (M) \cong \int_{\widehat{G}_K}^{\oplus} ( \mathcal{H}_{\pi} )^{K} d \nu (\pi)
\end{align}
accordingly, where $\nu$ is a positive measure on $\widehat{G}$, see for instance \cite{bunke2000spectrum}. Recall that $( \mathcal{H}_{\pi} )^{K} = \mathbb{C} e_{\pi}$ is one-dimensional for every $\pi \in \widehat{G}_K$. Recall moreover that, in rank one, $\widehat{G}_K$ is parametrized by a subset of $\mathbb{C} / \pm 1$. Specifically, $\widehat{G}_K$ consists of
\begin{itemize}[leftmargin=*]
\item the unitary spherical principal series $\pi_{\pm \lambda}$ ($\lambda \in \mathbb{R} / \pm 1$),
\item the trivial representation $\pi_{\pm i \rho} = 1$,
\item the complementary series $\pi_{\pm i \lambda}$ ($\lambda \in I$), where
\begin{equation*}
I =
\begin{cases}
\left(0, \rho \right) \quad if \ X=H^n (\mathbb{R}) \ or \ H^n (\mathbb{C}), \\
\left(0, \frac{m_{\alpha}}{2}+1 \right] \quad if \ X=H^n (\mathbb{H}) \ or \ H^2 (\mathbb{O}).
\end{cases}
\end{equation*}
\end{itemize}
This result goes back to \cite{kostant1969existence}. Under the assumption $\delta (\Gamma) \le \rho$, we know that $\lambda_0 = \rho^2$ is the bottom of the spectrum of $- \Delta$ on $L^2(M)$. As $- \Delta$ acts on $(\mathcal{H}_{{\pi}} )^{K} $ by multiplication by $\lambda^2 + \rho^2$, we deduce that \eqref{direct integral} involves only tempered representations, i.e., representations $\pi_{\lambda}$ with $\lambda \in \mathbb{R} / \pm 1$. Moreover, as the right convolution by $\psi \in \mathcal{S}(K \backslash G / K)$ acts on $(\mathcal{H}_{{\pi}_{\lambda}} )^{K} $ by multiplication by
\begin{align*}
\mathcal{H} f (\lambda) = \int_{G} f(x) \varphi_{\lambda} (x) dx,
\end{align*}
where $\varphi_{\lambda} (x) = \left\langle {\pi_{\lambda}(x) e_{{\pi}_{\lambda}},e_{{\pi}_{\lambda}} } \right\rangle$  is the spherical function \eqref{eigenfunction}, we deduce from \eqref{spherical function on noncompact spaces} that
\begin{align*}
\| . * \psi \|_{L^2(M) \rightarrow L^2(M)} 
\le \sup_{\lambda \in \mathbb{R}} \left| \int_{G} \psi(x) \varphi_{\lambda}(x) dx \right| 
\le \int_{G} | \psi (x)| \varphi_{0}(x) dx.
\end{align*}
\end{proof}

The following two lemmas are used in the proof of dispersive estimates.

\begin{lemma}\label{estimate of Poincaré series}
If $\Gamma$ is convex cocompact, then there exists a constant $C>0$ such that for all $x,y \in X$,
\begin{align*}
P(s;x,y) \le C P(s;\mathbf{0},\mathbf{0}),
\end{align*}
where $\mathbf{0} = eK$ denotes the origin of $X$.
\end{lemma}

\begin{proof}
Let $\conv (\Lambda_{\Gamma})$ be the convex hull of the limit set $\Lambda_{\Gamma}$ of $\Gamma$. Recall that $\Gamma$ is said to be convex cocompact if $\Gamma \backslash \conv (\Lambda_{\Gamma})$ is compact. Let $F$ be a compact fundamental domain containing $\mathbf{0}$ for the action of $\Gamma$ on $\conv (\Lambda_{\Gamma})$. Then, for each $z \in \conv (\Lambda_{\Gamma})$, there exists $\gamma \in \Gamma$ and $z' \in F$ such that $z= \gamma z'$.

The orthogonal projection $\pi_{\bot} : X \rightarrow \conv (\Lambda_{\Gamma})$ is defined as follows. For every $x$ in $X$, $\pi_{\bot}(x)$ is the unique point in $\conv (\Lambda_{\Gamma})$ such that
\begin{align*}
d(x, \pi_{\bot}(x)) = d(x, \conv (\Lambda_{\Gamma})) :=\inf_{y \in \conv (\Lambda_{\Gamma})} d(x,y).
\end{align*}
Then, for all $x,y \in X$, we have (see \cite{bridson1999metric}, Chap II, Proposition 2.4.)
\begin{align*}
d(\pi_{\bot}(x), \pi_{\bot}(y)) \le d(x,y).
\end{align*}
On the other hand, for all $x \in X$ and $\gamma \in \Gamma$,
\begin{align*}
d(\gamma x, \pi_{\bot}(\gamma x))
= d(\gamma x, \conv (\Lambda_{\Gamma}))
= d(x, \conv (\Lambda_{\Gamma})),
\end{align*}
since $\Lambda_{\Gamma}$ and $\conv (\Lambda_{\Gamma})$ are $\Gamma$-invariant. Thus
\begin{align*}
d(\gamma x, \pi_{\bot}(\gamma x))
= d(x,\pi_{\bot}(x)) 
= d(\gamma x, \gamma \pi_{\bot}(x)),
\end{align*}
which implies that, for all $x \in X$ and $\gamma \in \Gamma$, $\pi_{\bot}(\gamma x) = \gamma \pi_{\bot}(x)$. Therefore, for every $x,y \in X$ and $s>0$, the Poincaré series satisfies:
\begin{align}\label{poincare series 1}
\begin{split}
P(s;x,y) = \sum_{\gamma \in \Gamma} e^{-s d(x, \gamma y)}
\le& \sum_{\gamma \in \Gamma} e^{-s d(\pi_{\bot}(x), \pi_{\bot}(\gamma y))} \\
=& \sum_{\gamma \in \Gamma} e^{-s d(\pi_{\bot}(x), \gamma \pi_{\bot}(y))} 
=P(s;\pi_{\bot}(x), \pi_{\bot}(y)),
\end{split}
\end{align}
with $\pi_{\bot}(x), \pi_{\bot}(y) \in \conv (\Lambda_{\Gamma})$. Moreover, as there exist $\gamma_1,\gamma_2 \in \Gamma$ and $x',y' \in F$ such that $\pi_{\bot}(x) = \gamma_1 x'$, $\pi_{\bot}(y) = \gamma_2 y'$, we have
\begin{align}\label{poincare series 2}
P(s;\pi_{\bot}(x), \pi_{\bot}(y)) = \sum_{\gamma \in \Gamma} e^{-s d(x', \gamma_{1}^{-1} \gamma \gamma_{2} y')}
= \sum_{\gamma' \in \Gamma}e^{-s d(x', \gamma' y')}
= P(s;x',y').
\end{align}
Since $x',y' \in F$, the triangular inequality yields
\begin{align*}
d (\mathbf{0}, \gamma' \mathbf{0}) \le d (\mathbf{0}, x') + d(x', \gamma' y') + d (\gamma' y', \gamma' \mathbf{0}) \le d(x', \gamma' y') + 2 \diam (F).
\end{align*}
Hence
\begin{align}\label{poincare series 3}
P(s;x',y') \le e^{2s \diam (F)} P(s; \mathbf{0} , \mathbf{0}).
\end{align}
We conclude by combining \eqref{poincare series 1}, \eqref{poincare series 2} and \eqref{poincare series 3}.
\end{proof}

Consider the radial weight function defined by
\begin{align*}
\mu(x) = e^{(\delta(\Gamma) + \varepsilon) d(x, \mathbf{0})},
\end{align*}
with $0 < \varepsilon < \rho - \delta(\Gamma)$. We prove the following lemma by applying previous results.

\begin{lemma}\label{B(f,g)}
Let $f$ be a reasonable function on $M$, and $g$ be a radial reasonable  function on $X$. Then the bilinear operator $B(f,g) := f* (\mu^{-1}g)$ satisfies the following estimate:
\begin{align*}
\| B( \cdot,g) \|_{L^{q'}(M) \rightarrow L^q(M)} \le C_q  \left( \int_{G} \varphi_0(x) \mu^{-1}(x) |g(x)|^{q/2} dx \right)^{2/q},
\end{align*}
for all $2 \le q \le \infty$.
\end{lemma}

\begin{proof}
According to Proposition \ref{Kunze-Stein},
\begin{align}\label{2-2}
\| B(\cdot,g) \|_{L^{2}(M) \rightarrow L^2(M)} \le \int_{G} \varphi_0(x) \mu^{-1}(x) |g(x)| dx.
\end{align}

Since $f$ is left $\Gamma$-invariant, we can rewrite
\begin{align*}
B(f,g)(x) =& \int_{X} \left( \mu^{-1} g \right) \left(y^{-1}x \right) f(y) dy \\
=& \int_{\Gamma \backslash X} \left( \sum_{\gamma \in \Gamma} \left( \mu^{-1} g \right) \left(y^{-1} \gamma x \right) \right) f(y) dy,
\end{align*}
with
\begin{align*}
\left| \sum_{\gamma \in \Gamma} \left( \mu^{-1} g \right) \left(y^{-1} \gamma x \right) \right| 
\le& ||g||_{\infty} \sum_{\gamma \in \Gamma}e^{-(\delta(\Gamma) + \varepsilon) d(x, \gamma y)} \\
\le& ||g||_{\infty} P_{\delta(\Gamma)+ \varepsilon} (x,y) 
\le ||g||_{\infty} P_{\delta(\Gamma)+ \varepsilon} (0,0),
\end{align*}
according to Lemma \ref{estimate of Poincaré series}. Hence
\begin{align}\label{1-infty}
\| B(\cdot,g) \|_{L^{1}(M) \rightarrow L^{\infty}(M)} 
= \sup_{x,y \in G} \left| \sum_{\gamma \in \Gamma} \left( \mu^{-1} g \right) \left(y^{-1} \gamma x \right) \right| \le C ||g||_{\infty}.
\end{align}
We conclude by standard interpolations between \eqref{2-2} and \eqref{1-infty}.
\end{proof}

We prove now our main result.

\begin{proof}[Proof of Theorem \ref{dispersive estimate}]
We split up the proof into two parts, depending whether the time $t$ is small or large.
\par
\noindent {\bf Dispersive estimate for small time} \par
Assume that $0 < |t| <1$. On the one hand, by using the Lemma \ref {B(f,g)} with $g(x) = \mu(x) \omega_{t}^{\sigma,0}(x)$, we have
\begin{align*}
\| \cdot * \omega_{t}^{\sigma,0} \|_{L^{q'}(M) \rightarrow L^{q}(M)} \le C_q \left( \int_{G} \varphi_0(x) \mu(x)^{\frac{q}{2}-1} |\omega_{t}^{\sigma,0}(x)|^{\frac{q}{2}} dx \right)^{\frac{2}{q}}.
\end{align*}
Notice that the ground spherical function $\varphi_0$, the weight $\mu$ and the kernel $ \omega_{t}^{\sigma,0}$ are all $K$-bi-invariant. By using the expression \eqref{Cartan decomposition for Haar measure} of the Haar measure in the Cartan decomposition, together with the estimate \eqref{estimate of Haar measure}, we obtain fisrt
\begin{align*}
\int_{G} \varphi_0(x) \mu(x)^{\frac{q}{2}-1} |\omega_{t}^{\sigma,0}(x)|^{\frac{q}{2}} dx 
\lesssim& \int_{0}^{+ \infty} \varphi_0 (r) \mu(r)^{\frac{q}{2}-1} | \omega_{t}^{\sigma,0}(r)|^{\frac{q}{2}} e^{2 \rho r} dr. 
\end{align*}
As $| \omega_{t}^{\sigma,0}(r)| \lesssim \varphi_0(r)$, according to  Theorem \ref{pointwise kernel estimate}, and
\footnote{The symbol $\asymp$ means that there exist two constants $0 < C_1 \le C_2 < + \infty$ such that
\begin{align*}
C_1 \le \frac{\varphi_0 (r)}{(1+r)e^{-\rho r}} \le C_2, \ \forall r \ge 0.
\end{align*}
}
 $\varphi_0 (r) \asymp (1+r)e^{-\rho r}$, we obtain next
\begin{align*}
\int_{0}^{+ \infty} \varphi_0 (r) \mu(r)^{\frac{q}{2}-1} | \omega_{t}^{\sigma,0}(r)|^{\frac{q}{2}} e^{2 \rho r} dr \lesssim \int_{0}^{+ \infty}(1+r)^{\frac{q}{2}+1} e^{- \left(\frac{q}{2}-1 \right) \left( \rho - \delta (\Gamma) - \varepsilon \right)r} dr.
\end{align*}
Since $ \rho - \delta (\Gamma) - \varepsilon >0$, the last integral is finite for any $2 < q < + \infty$. By using Lemma \ref{B(f,g)}, we conclude that
\begin{align*}
\sup_{0 < |t| < 1} \| \cdot * \omega_{t}^{\sigma,0} \|_{L^{q'}(M) \rightarrow L^{q}(M)} < + \infty.
\end{align*}
On the other hand, consider the analytic family of operators $\widetilde{W}_{t}^{\sigma, \infty}$ defined by \eqref{analytic operator}. When $\re \sigma = 0$, the spectral theorem yields
\begin{align}\label{spectral theory}
\left\| \widetilde{W}_{t}^{\sigma, \infty} \right\|_{L^2 (M) \rightarrow L^2 (M)} \lesssim \left\| e^{itD} \right\|_{L^2 (M) \rightarrow L^2 (M)} =1,
\end{align}
for all $t \in \mathbb{R}^{*}$. When $\re \sigma = \frac{n+1}{2}$, Theorem \ref{pointwise kernel estimate} yields
\begin{align*}
\left\| \widetilde{W}_{t}^{\sigma, \infty} \right\|_{L^1 (M) \rightarrow L^{\infty} (M)} = \sup_{x,y \in M} \left\| \widehat{\tilde{\omega}_{t}^{\sigma}} (x,y)\right\| \lesssim \left\| \mu \tilde{\omega}_{t}^{\sigma} (x,y) \right\|_{\infty} \lesssim |t|^{-\frac{n-1}{2}}
\end{align*}
in dimension $n \ge 3$. By applying Stein's interpolation theorem for an analytic family of operators, we obtain
\begin{align*}
\left\| \widetilde{W}_{t}^{\frac{n+1}{2} (1-\theta), \infty} \right\|_{L^{q'} (M) \rightarrow L^{q} (M)} \lesssim |t|^{-\frac{n-1}{2}(1- \theta)},
\end{align*}
where $\theta = \frac{2}{q}$, that is
\begin{align*}
\left\| \cdot * \widetilde{\omega}_{t}^{\sigma, \infty} \right\|_{L^{q'} (M) \rightarrow L^{q} (M)} \lesssim |t|^{-(n-1)(\frac{1}{2} - \frac{1}{q})},
\end{align*}
with $\sigma = (n+1)(\frac{1}{2} - \frac{1}{q})$. In conclusion,
\begin{align*}
\left\| \widehat{W_{t}^{\sigma}} \right\|_{L^{q'}(M) \rightarrow L^q(M)}
\lesssim  |t|^{-(n-1)\left(\frac{1}{2} - \frac{1}{q}\right)}, \ \forall 0<|t|<1
\end{align*}
for $n \ge 3$, $\sigma = (n+1)(\frac{1}{2} - \frac{1}{q})$ and $2<q< \infty$. In dimension $n=2$, the same arguments yield
\begin{align*}
\left\| \widehat{W_{t}^{\sigma}} \right\|_{L^{q'}(M) \rightarrow L^q(M)}
\lesssim  |t|^{-\left(\frac{1}{2} - \frac{1}{q}\right)} (1- \log |t| )^{1-\frac{2}{q}}, \ \forall 0<|t|<1
\end{align*}
for $\sigma = 3 \left(\frac{1}{2} - \frac{1}{q}\right)$ and $2 < q < + \infty$.
\par
\noindent {\bf Dispersive estimate for large time} \par
Assume now that $|t| \ge 1$. We proceed as before after splitting up  the kernel as follows:
\begin{align*}
{\omega}_{t}^{\sigma}  = \mathds{1}_{B \left( 0, \frac{|t|}{2}\right)} \omega_{t}^{\sigma,0} + \mathds{1}_{X \backslash B \left( 0, \frac{|t|}{2}\right)} \omega_{t}^{\sigma,0} + {\omega}_{t}^{\sigma, \infty}.
\end{align*}
By using Lemma \ref{B(f,g)} and Theorem \ref{pointwise kernel estimate}, we obtain
\begin{align*}
\left\| \cdot * \mathds{1}_{B \left( 0, \frac{|t|}{2}\right)} \omega_{t}^{\sigma,0} \right\|_{L^{q'}(M) \rightarrow L^q(M)} \lesssim& \left\lbrace \int_{0}^{\frac{|t|}{2}} \varphi_0 (r) \mu(r)^{\frac{q}{2}-1} |\omega_{t}^{\sigma,0}(r)|^{\frac{q}{2}}e^{2 \rho r} dr \right\rbrace^{\frac{2}{q}} \\
\lesssim& |t|^{-\frac{3}{2}} \underbrace{ \left\lbrace \int_{0}^{\frac{|t|}{2}} (1+r)^{1+q} e^{- \left( \frac{q}{2} - 1 \right) \left( \rho - \delta (\Gamma) - \varepsilon \right)r} dr \right\rbrace^{\frac{2}{q}}}_{< + \infty}
\end{align*}
and 
\begin{align*}
\left\| f *\mathds{1}_{X \backslash B \left( 0, \frac{|t|}{2}\right)} \omega_{t}^{\sigma,0} \right\|_{L^{q'}(M) \rightarrow L^q(M)} \lesssim& \left\lbrace \int_{\frac{|t|}{2}}^{+ \infty} \varphi_0 (r) \mu(r)^{\frac{q}{2}-1} |\omega_{t}^{\sigma,0}(r)|^{\frac{q}{2}}e^{2 \rho r} dr \right\rbrace^{\frac{2}{q}} \\
\lesssim& \left\lbrace \int_{\frac{|t|}{2}}^{+ \infty} (1+r)^{\frac{q}{2}+1} e^{- \left( \frac{q}{2} - 1 \right) \left( \rho - \delta (\Gamma) - \varepsilon \right) r} dr \right\rbrace^{\frac{2}{q}},
\end{align*}
which is $O(|t|^{-N})$, for any $N > 0$. 
Instead of ${\omega}_{t}^{\sigma, \infty}$, we consider again the kernel $\widetilde{\omega}_{t}^{\sigma, \infty}$. By Theorem \ref{pointwise kernel estimate}, the associated operators satisfy
\begin{align*}
\left\| \widetilde{W}_{t}^{\sigma, \infty} \right\|_{L^1 (M) \rightarrow L^{\infty} (M)} \lesssim |t|^{-N}, \ \forall N \in \mathbb{N}
\end{align*}
when $\re \sigma = \frac{n+1}{2}$. By using again Stein's interpolation theorem and by summing up these estimates, we obtain finally
\begin{align*}
\left\| \widehat{W_{t}^{\sigma}} \right\|_{L^{q'}(M) \rightarrow L^q(M)}
\lesssim  |t|^{-\frac{3}{2}}, \ \forall |t| \ge 1
\end{align*}
for $n \ge 2$, $\sigma = (n+1)(\frac{1}{2} - \frac{1}{q})$ and $2<q< \infty$.
\end{proof}
\section{Strichartz estimate and applications}\label{section applcations}
Let $\sigma \in \mathbb{R}$ and $1<q< \infty$. Recall that the Sobolev  space $H^{\sigma,q}(M)$ is the image of $L^q(M)$ under the operator $(- \Delta)^{- \frac{\sigma}{2}}$, equipped with the norm
\begin{align*}
\| f \|_{H^{\sigma,q}(M)} = \| (- \Delta)^{ \frac{\sigma}{2}} f \|_{L^q(M)}.
\end{align*}
If  $\sigma=N$ is a nonnegative integer, then $H^{\sigma,q}(M)$ coincides with the classical Sobolev space
\begin{align*}
W^{N,q} (M) = \left\lbrace f \in L^q(M) \ | \ \nabla^{j} f \in L^q(M), \ \forall 1 \le j \le N \right\rbrace,
\end{align*}
defined by means of covariant derivatives. The following Sobolev embedding theorem is used in next subsection:
\begin{theorem}
Let $1< q_1,q_2 < \infty$ and $\sigma_1, \sigma_2 \in \mathbb{R}$ such that $\sigma_1 - \sigma_2 \ge \frac{n}{q_1} - \frac{n}{q_2} \ge 0$. Then
\begin{align}\label{embedding}
H^{\sigma_1,q_1}(M) \subset H^{\sigma_2,q_2}(M).
\end{align}
\end{theorem}
We refer to \cite{triebel1992theory} for more details about function spaces on Riemannian manifolds. Let us state next the Strichartz estimate and some applications. The proofs are straightforwardly adapted from \cite{anker2014wave} and are therefore omitted.
\subsection{Strichartz estimate}\label{section Strichartz} { \ } \par
Recall the linear inhomogenous Klein-Gordon equation on $M$:
\begin{align}\label{Cauchy problem}
\begin{cases}
\partial_{t}^2 u(t,x) + D^2 u(t,x) = F(t,x), \\
u(0,x) =f(x) , \ \partial_{t}|_{t=0} u(t,x) =g(x).
\end{cases}
\end{align}
whose solution is given by Duhamel's formula:
\begin{align*}
u(t,x) 
= \left( \cos tD \right) f(x)
+ \frac{\sin tD}{D} g(x)
+ \int_{0}^{t} \frac{\sin (t-s) D}{D} F(s,x)ds.
\end{align*}

We consider first the case $n \ge 4$ and discuss the $2$-dimensional and $3$-dimensional cases in the final remarks. Recall that a couple $(p,q)$ is called admissible if $\left( \frac{1}{p}, \frac{1}{q} \right)$ belongs to the triangle
\begin{align*}
\left\lbrace \left( \frac{1}{p}, \frac{1}{q} \right) \in \left( 0, \frac{1}{2} \right) \times \left( 0, \frac{1}{2} \right) \ \Big| \ \frac{1}{p} \ge \frac{n-1}{2} \left( \frac{1}{2} - \frac{1}{q} \right) \right\rbrace \bigcup \left\lbrace  \left( 0, \frac{1}{2} \right),  \left( \frac{1}{2}, \frac{1}{2} - \frac{1}{n-1} \right) \right\rbrace.
\end{align*}
\vspace{-0.5cm}
\begin{figure}[!h]\label{n >4}
\centering
\begin{tikzpicture}[scale=0.47][line cap=round,line join=round,>=triangle 45,x=1.0cm,y=1.0cm]
\draw[->,color=black] (0.,0.) -- (8.805004895364174,0.);
\foreach \x in {,2.,4.,6.,8.,10.}
\draw[shift={(\x,0)},color=black] (0pt,-2pt);
\draw[color=black] (8,0.08808504628984362) node [anchor=south west] {\large $\frac{1}{p}$};
\draw[->,color=black] (0.,0.) -- (0.,8);
\foreach \y in {,2.,4.,6.,8.}
\draw[shift={(0,\y)},color=black] (-2pt,0pt);
\draw[color=black] (0.11010630786230378,8) node [anchor=west] {\large $\frac{1}{q}$};
\clip(-5,-2) rectangle (11.805004895364174,9.404473957611293);
\fill[line width=2.pt,color=ffqqqq,fill=ffqqqq,fill opacity=0.15000000596046448] (0.45855683542994374,5.928142080369191) -- (5.85410415683891,5.8847171522290775) -- (5.864960388873938,4.136863794589543) -- cycle;
\draw [line width=1.pt,dash pattern=on 5pt off 5pt,color=ffqqqq] (0.,6.)-- (6.,6.);
\draw [line width=1.pt,dash pattern=on 5pt off 5pt,color=ffqqqq] (6.,6.)-- (6.,4.);
\draw [line width=1.pt,color=ffqqqq] (6.,4.)-- (0.,6.);
\draw [line width=0.8pt] (6.,4.)-- (6.,0.);
\draw [->,line width=0.5pt,color=ffqqqq] (1.8569381602310457,-0.8701903243695805) -- (2.891937454136701,4.811295161325333);
\begin{scriptsize}
\draw [fill=ffqqqq] (0.,6.) circle (5pt);
\draw[color=black] (-0.5709441083587729,6) node {\large $\frac{1}{2}$};
\draw[color=black] (-1.5,4) node {\large $\frac{1}{2}- \frac{1}{n-1}$};
\draw[color=black] (-0.5,-0.2) node {\large $0$};
\draw[color=black] (6.2,-0.7) node {\large $\frac{1}{2}$};
\draw[color=ffqqqq] (1.8,-1.3) node {\large $\frac{1}{p} = \frac{n-1}{2} \left( \frac{1}{2} - \frac{1}{q} \right)$};
\draw [color=ffqqqq] (6.,6.) circle (5pt);
\draw [fill=ffqqqq] (6.,4.) circle (5pt);
\end{scriptsize}
\end{tikzpicture}
\caption{Admissibility in dimension $n \ge 4$.}
\end{figure}
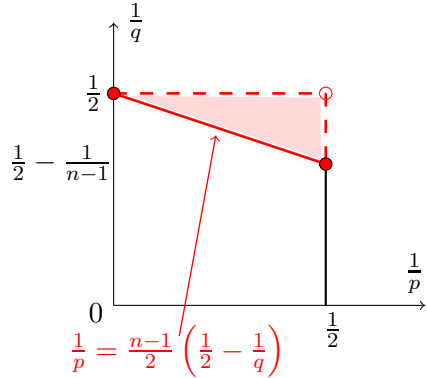

\begin{theorem}\label{thm strichartz}
Let $(p,q)$ and $(\tilde{p}, \tilde{q})$ be two admissible couples, and let
\begin{align*}
\sigma \ge \frac{n+1}{2} \left( \frac{1}{2} - \frac{1}{q} \right) \ and \ \widetilde{\sigma} \ge \frac{n+1}{2} \left( \frac{1}{2} - \frac{1}{\tilde{q}} \right).
\end{align*}
Then all solutions $u$ to the Cauchy problem \eqref{Cauchy problem} satisfy the  following Strichartz estimate:
\begin{align}\label{Strichartz}
\left\| \nabla_{\mathbb{R} \times M} u \right\|_{L^p(I; H^{-\sigma,q}(M))} \lesssim \left\| f \right\|_{H^1(M)} + \left\| g \right\|_{L^2(M)} + \left\| F \right\|_{L^{\tilde{p}'}(I; H^{\tilde{\sigma},\tilde{q}'}(M))}.
\end{align}
\end{theorem}

\begin{remark}
In comparison with hyperbolic spaces , observe that we loose the right edge of the admissible triangle. The reason is that the standard $TT^*$ method used to prove the Strichartz estimate breaks down in the critical case where $p=2$ and $q < 2\frac{n-1}{n-3}$. The dyadic decomposition method carried out in \cite{keel1998endpoint} takes care of the endpoints, but it requires a stronger dispersive property than Theorem \ref{dispersive estimate} in small time, which reads
\begin{align*}
\left\| \widehat{W_{t}^{\sigma}} \right\|_{L^{\tilde{q}'}(M) \rightarrow L^q(M)} \lesssim
|t|^{-(n-1) \max \left( \frac{1}{2} - \frac{1}{q}, \frac{1}{2} - \frac{1}{\tilde{q}} \right)}, \ \forall \ 0< |t| < 1
\end{align*}
for $n \ge 3$, $2 < q, \tilde{q} < \infty$ and $\sigma \ge (n+1) \max \left( \frac{1}{2} - \frac{1}{q}, \frac{1}{2} - \frac{1}{\tilde{q}} \right)$. Such an estimate would follow from
\begin{align*}
\left\| \widehat{\omega_{t}^{\sigma}} \right\|_{L^q(M)} \lesssim |t|^{- \frac{n-1}{2}}, \ \forall 0 < |t| <1
\end{align*}
for $\sigma \ge \frac{n+1}{2} \left( \frac{1}{2} - \frac{1}{q} \right)$ and $2 < q < + \infty$, which is unknown so far.

However, these critical points are not relevant for the following well-posedness problems, hence we obtain the same results as on real hyperbolic spaces. The admissible range in \eqref{Strichartz} can be widen by using the Sobolev embedding theorem.
\end{remark}

\begin{corollary}\label{corollary strichartz}
Let $(p,q)$ and $(\tilde{p}, \tilde{q})$ be two couples corresponding to the square
\begin{align*}
\left[ 0, \frac{1}{2} \right) \times \left(0, \frac{1}{2} \right) \bigcup \left\lbrace \left( 0, \frac{1}{2} \right), \left( \frac{1}{2}, \frac{1}{2} - \frac{1}{n-1} \right) \right\rbrace,
\end{align*}
\vspace{-0.5cm}
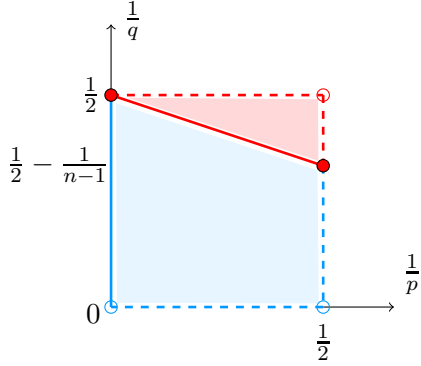
\begin{figure}[!h]
\centering
\begin{tikzpicture}[scale=0.47][line cap=round,line join=round,>=triangle 45,x=1.0cm,y=1.0cm]
\draw[->,color=black] (6.,0.) -- (8,0.);
\foreach \x in {,1.,2.,3.,4.,5.,6.,7.,8.,9.,10.,11.}
\draw[shift={(\x,0)},color=black] (0pt,-2pt);
\draw[color=black] (8,0.1) node [anchor=south west] {\large $\frac{1}{p}$};
\draw[->,color=black] (0.,6.) -- (0.,8);
\foreach \y in {,1.,2.,3.,4.,5.,6.,7.}
\draw[shift={(0,\y)},color=black] (-2pt,0pt);
\draw[color=black] (0.1,8.1) node [anchor=west] {\large $\frac{1}{q}$};
\clip(-4,-2) rectangle (11.868487651785372,7.652018714581741);
\fill[line width=2.pt,color=ffqqqq,fill=ffqqqq,fill opacity=0.15000000596046448] (0.45855683542994374,5.928142080369191) -- (5.85410415683891,5.8847171522290775) -- (5.864960388873938,4.136863794589543) -- cycle;
\fill[line width=2.pt,color=qqzzff,fill=qqzzff,fill opacity=0.10000000149011612] (0.13636385492209063,5.773698170331237) -- (5.837359246154291,3.8701454719028483) -- (5.876010062366035,0.10169089125781497) -- (0.16535196708089844,0.12101629936368694) -- cycle;
\draw [line width=1.pt,dash pattern=on 3pt off 3pt,color=ffqqqq] (0.,6.)-- (6.,6.);
\draw [line width=1.pt,dash pattern=on 3pt off 3pt,color=ffqqqq] (6.,6.)-- (6.,4.);
\draw [line width=1.pt,color=ffqqqq] (6.,4.)-- (0.,6.);
\draw [line width=1.pt,color=qqzzff] (0.,6.)-- (0.,0.);
\draw [line width=1.pt,dash pattern=on 3pt off 3pt,color=qqzzff] (6.,4.)-- (6.,0.);
\draw [line width=1.pt,dash pattern=on 3pt off 3pt,color=qqzzff] (6.,0.)-- (0.,0.);
\begin{scriptsize}
\draw [fill=ffqqqq] (0.,6.) circle (5pt);
\draw[color=black] (-0.5709441083587729,6) node {\large $\frac{1}{2}$};
\draw [color=ffqqqq] (6.,6.) circle (5pt);
\draw [fill=ffqqqq] (6.,4.) circle (5pt);
\draw [color=qqzzff] (0.,0.) circle (5pt);
\draw[color=black] (-0.5,-0.2) node {\large $0$};
\draw [color=qqzzff] (6.,0.) circle (5pt);
\draw[color=black] (6,-1) node {\large $\frac{1}{2}$};
\draw[color=black] (-1.5,4) node {\large $\frac{1}{2}- \frac{1}{n-1}$};
\end{scriptsize}
\end{tikzpicture}
\vspace{-0.5cm}
\caption{Case $n \ge 4$}
\end{figure}

\noindent
Let $\sigma, \tilde{\sigma} \in \mathbb{R}$ such that $\sigma \ge \sigma(p,q)$, where
\begin{align*}
\sigma(p,q) =& \frac{n+1}{2} \left( \frac{1}{2} - \frac{1}{q} \right) + \max \left\lbrace 0, \frac{n-1}{2} \left( \frac{1}{2} - \frac{1}{q} \right) - \frac{1}{p} \right\rbrace,
\end{align*}
and similarly $\tilde{\sigma} \ge \sigma(\tilde{p}, \tilde{q})$. Then the Strichartz estimate \eqref{Strichartz} holds for all solutions to the Cauchy problem \eqref{Cauchy problem}.
\end{corollary}

\begin{remark}
Theorem \ref{thm strichartz} and Corollary \ref{corollary strichartz} still hold true in lower dimension $n=3$ and $n=2$ with similar proofs. In particular, the endpoint $(p,q) = (2, \infty)$ is excluded and the admissible set in dimension $2$ becomes
\begin{align*}
\left\lbrace \left( \frac{1}{p}, \frac{1}{q} \right) \in \left( 0, \frac{1}{2} \right) \times \left( 0, \frac{1}{2} \right) \ \Big| \ \frac{1}{p} > \frac{1}{2} \left( \frac{1}{2} - \frac{1}{q} \right) \right\rbrace \bigcup \left\lbrace  \left( 0, \frac{1}{2} \right) \right\rbrace,
\end{align*}
and the region in Corollary \ref{corollary strichartz} is
\begin{align*}
\left\lbrace \left( \frac{1}{p}, \frac{1}{q} \right) \in \left( 0, \frac{1}{4} \right) \times \left( 0, \frac{1}{2} \right) \ \Big| \ \frac{1}{p} \le \frac{1}{2} \left( \frac{1}{2} - \frac{1}{q} \right) \right\rbrace.
\end{align*}
\end{remark}

\vspace{-0.46cm}
\begin{figure}[!h]
\begin{subfigure}[t]{0.5\textwidth}
\centering
\begin{tikzpicture}[scale=0.5][line cap=round,line join=round,>=triangle 45,x=1.0cm,y=1.0cm]
\draw[->,color=black] (6.,0.) -- (8,0.);
\foreach \x in {,1.,2.,3.,4.,5.,6.,7.,8.,9.,10.,11.,12.,13.}
\draw[shift={(\x,0)},color=black] (0pt,-2pt);
\draw[color=black] (8,0.1) node [anchor=south west] {\large $\frac{1}{p}$};
\draw[->,color=black] (0.,0.) -- (0.,8);
\foreach \y in {,1.,2.,3.,4.,5.,6.,7.,8.}
\draw[shift={(0,\y)},color=black] (-2pt,0pt);
\draw[color=black] (0.1,8.1) node [anchor=west] {\large $\frac{1}{q}$};
\clip(-2,-2) rectangle (13.119347003417166,8.080526271073529);
\fill[line width=2.pt,color=ffqqqq,fill=ffqqqq,fill opacity=0.15000000596046448] (0.23947996857607806,5.9156932292089195) -- (5.85410415683891,5.8847171522290775) -- (5.896883481046051,0.25828971673893225) -- cycle;
\fill[line width=2.pt,color=qqzzff,fill=qqzzff,fill opacity=0.10000000149011612] (0.08350079729924402,5.793380424802772) -- (0.11455249049619687,0.12106762660967957) -- (5.704457330452569,0.11429587297972885) -- cycle;
\draw [line width=1.pt,dash pattern=on 3pt off 3pt,color=ffqqqq] (0.,6.)-- (6.,6.);
\draw [line width=1.pt,dash pattern=on 3pt off 3pt,color=ffqqqq] (6.,6.)-- (6.,0.);
\draw [line width=1.pt,color=ffqqqq] (6.,0.)-- (0.,6.);
\draw [line width=1.pt,color=qqzzff] (0.,6.)-- (0.,0.);
\draw [line width=1.pt,dash pattern=on 3pt off 3pt,color=qqzzff] (6.,0.)-- (0.,0.);
\begin{scriptsize}
\draw [fill=ffqqqq] (0.,6.) circle (5pt);
\draw[color=black] (-0.5709441083587729,6) node {\large $\frac{1}{2}$};
\draw [color=ffqqqq] (6.,6.) circle (5pt);
\draw [color=qqzzff] (6.,0.) circle (5pt);
\draw [color=qqzzff] (0.,0.) circle (5pt);
\draw[color=black] (-0.5,-0.2) node {\large $0$};
\draw[color=black] (6,-1) node {\large $\frac{1}{2}$};
\end{scriptsize}
\end{tikzpicture}
\end{subfigure}
~ 
\begin{subfigure}[t]{0.5\textwidth}
\centering
\begin{tikzpicture}[scale=0.5][line cap=round,line join=round,>=triangle 45,x=1.0cm,y=1.0cm]
\draw[->,color=black] (6.,0.) -- (8,0.);
\foreach \x in {,1.,2.,3.,4.,5.,6.,7.,8.,9.,10.,11.,12.,13.}
\draw[shift={(\x,0)},color=black] (0pt,-2pt);
\draw[color=black] (8,0.1) node [anchor=south west] {\large $\frac{1}{p}$};
\draw[->,color=black] (0.,0.) -- (0.,8);
\foreach \y in {,1.,2.,3.,4.,5.,6.,7.}
\draw[shift={(0,\y)},color=black] (-2pt,0pt);
\draw[color=black] (0.1,8.1) node [anchor=west] {\large $\frac{1}{q}$};
\clip(-4,-2) rectangle (13.0320877855373,7.819160947044588);
\fill[line width=2.pt,color=qqzzff,fill=qqzzff,fill opacity=0.10000000149011612] (0.049288803201222935,5.589067834924169) -- (0.09973660476186558,0.1051097243515092) -- (2.8205713524874143,0.10510972435150921) -- cycle;
\fill[line width=2.pt,color=ffqqqq,fill=ffqqqq,fill opacity=0.10000000149011612] (0.16242526408151014,5.876115136959884) -- (5.888713400893983,5.846660512380616) -- (5.88724066966502,0.1317269783881381) -- (3.1126305298318058,0.13982700014743682) -- cycle;
\draw [line width=1.pt,dash pattern=on 3pt off 3pt,color=ffqqqq] (0.,6.)-- (6.,6.);
\draw [line width=1.pt,dash pattern=on 3pt off 3pt,color=ffqqqq] (6.,6.)-- (6.,0.);
\draw [line width=1.pt,color=qqzzff] (0.,6.)-- (0.,0.);
\draw [line width=1.pt,color=qqzzff] (0.,6.)-- (3.,0.);
\draw [line width=1.pt,dash pattern=on 3pt off 3pt,color=qqzzff] (0.,0.)-- (3.,0.);
\draw [line width=1.pt,dash pattern=on 3pt off 3pt,color=ffqqqq] (3.,0.)-- (6.,0.);
\begin{scriptsize}
\draw [fill=ffqqqq] (0.,6.) circle (5pt);
\draw[color=black] (-0.5709441083587729,6) node {\large $\frac{1}{2}$};
\draw [color=ffqqqq] (6.,6.) circle (5pt);
\draw [color=ffqqqq] (6.,0.) circle (5pt);
\draw [color=qqzzff] (0.,0.) circle (5pt);
\draw[color=black] (-0.5,-0.2) node {\large $0$};
\draw[color=black] (6,-1) node {\large $\frac{1}{2}$};
\draw[color=black] (3,-1) node {\large $\frac{1}{4}$};
\draw [color=ffqqqq] (3.,0.) circle (5pt);
\end{scriptsize}
\end{tikzpicture}
\end{subfigure}
\vspace{-0.75cm}
\caption{Cases $n=3$ and $n=2$}
\end{figure}
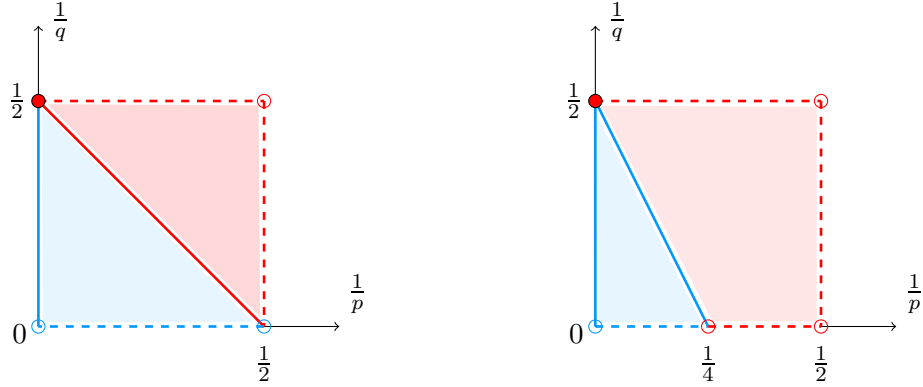

\subsection{Global well-posedness in $L^{p} \left( \mathbb{R}, L^q (M) \right)$}\label{section GWP} ~ \par
We refer to \cite{anker2014wave} for more detailed proofs of the following well-posedness results. By using the classical fixed point scheme with the previous Strichartz estimates, one obtains the global well-posedness for the semilinear equation
\begin{align}\label{semilinear equation}
\begin{cases}
\partial_{t}^2 u(t,x) + D^2 u(t,x) = F(u(t,x)), \\
u(0,x) =f(x) , \ \partial_{t}|_{t=0} u(t,x) =g(x).
\end{cases}
\end{align}
on $M$ with power-like nonlinearities $F$ satisfying
\begin{align*}
|F(u)| \le C |u|^{\gamma}, \quad |F(u)-F(v)| \le C \left(|u|^{\gamma-1} + |v|^{\gamma -1} \right) |u-v|, \quad \gamma >1.
\end{align*}
and small initial data $f$ and $g$. Assume that $n \ge 3$, and consider the following powers
\begin{align*}
\gamma_1 = 1+ \frac{3}{n}, \ \gamma_2 = 1 + \frac{2}{\frac{n-1}{2}+ \frac{2}{n-1}}, \ \gamma_c = 1+ \frac{4}{n-1}, \\
\gamma_3 = 
\begin{cases}
\frac{\frac{n+6}{2} + \frac{2}{n-1} + \sqrt{4n + \left( \frac{6-n}{2}+\frac{2}{n-1} \right)^2}}{n} \quad &if \ n \le 5, \\
1 + \frac{2}{\frac{n-1}{2} - \frac{1}{n-1}}\quad &if \ n \ge 6,
\end{cases} \\
\gamma_4 = 
\begin{cases}
1 + \frac{4}{n-2} \quad &if \ n \le 5, \\
\frac{n-1}{2} + \frac{3}{n+1} - \sqrt{\left( \frac{n-3}{2} + \frac{3}{n+1} \right)^2 - 4 \frac{n-1}{n+1}}  &if \ n \ge 6,
\end{cases}
\end{align*}
and the following curves
\begin{align*}
\sigma_1(\gamma) = \frac{n+1}{4} - \frac{(n+1)(n+5)}{8n} \frac{1}{\gamma - \frac{n+1}{2n}}, \\
\sigma_2(\gamma) = \frac{n+1}{4} - \frac{1}{\gamma -1}, \quad 
\sigma_3(\gamma) = \frac{n}{2} - \frac{2}{\gamma -1}.
\end{align*}
Denote by $0^{+}$ any small positive constant. In dimension $n \ge 3$, the equation \eqref{semilinear equation} is globally well-posed for small initial data in $H^{\sigma,2}(M) \times H^{\sigma - 1,2}(M)$ provided that
\begin{align}\label{GWP conditions}
\begin{cases}
\sigma = 0^{+}, \ &if \ 1< \gamma \le \gamma_1, \\
\sigma = \sigma_1 (\gamma), \ &if \ \gamma_1< \gamma \le \gamma_2, \\
\sigma = \sigma_2 (\gamma), \ &if \ \gamma_2< \gamma \le \gamma_c, \\
\sigma = \sigma_3 (\gamma), \ &if \ \gamma_c< \gamma \le \gamma_4,
\end{cases}
\end{align}

Similar results hold in dimension $2$, see \cite{anker2014wave}. Observe that one obtains the same global well-posedness results on $M$ as on real hyperbolic spaces, without further assumptions. In comparison with the Euclidean setting, this is a consequence of the larger admissible set for the Strichartz estimate.
\bibliographystyle{abbrv}
\bibliography{ZHANG - Wave and Klein-Gordon equations on locally symmetric spaces 2nd version - 2018 .bib}
\end{document}